\documentclass{amsart}
\newtheorem{lemma}{Lemma}
\newtheorem{definition}{Definition}
\newtheorem{theorem}{Theorem}
\newtheorem{proposition}{Proposition}
\usepackage{amsmath, amsthm, amsfonts, amssymb}

\title{Stochastic Homogenization of Parabolic Equations with Lower-order Terms}
\author{Man Yang}
\address{Kyushu University, Japan.}
\email{yang.man.365@s.kyushu-u.ac.jp}

\begin{document}

\begin{abstract}
The study of homogenization results has long been a central focus in the field of mathematical analysis, particularly for equations without lower-order terms. However, the importance of studying homogenization results for parabolic equations with lower-order terms cannot be understated. In this study, we aim to extend the analysis to homogenization for the general parabolic equation with random coefficients:
	\begin{equation*}
		\partial_{t}p^\epsilon-\nabla\cdot\left(\mathbf{a}\left( \dfrac{x}{\epsilon},\dfrac{t}{\epsilon^2}\right)\nabla p^\epsilon\right)-\mathbf{b}\left( \dfrac{x}{\epsilon},\dfrac{t}{\epsilon^2}\right)\nabla p^\epsilon -\mathbf{d}\left( \dfrac{x}{\epsilon},\dfrac{t}{\epsilon^2}\right) p^\epsilon=0.
	\end{equation*}
  Moreover, we establish the Caccioppoli inequality and Meyers estimate for the generalized parabolic equation. By using the generalized Meyers estimate, we get the weak convergence of $p^\epsilon$ in $H^1$.
\end{abstract}


\maketitle







\section{Introduction}

Homogenization theory studies the effects of high-frequency oscillations in the coefficients on the solutions of partial differential equations. 
The classical theory of homogenization for elliptic systems with periodic coefficients was established by the French, Italian, and Russian schools (e.g. see \cite{bensoussan2011asymptotic}).
Most of the previous works have dealt with periodic coefficients (e.g. see \cite{jikov2012homogenization}) because the qualitative homogenization theory can be made quantitative in the periodic setting. The qualitative homogenization of elliptic equations with random coefficients was first established by Kozlov  \cite{kozlov1979averaging} and the quantitative results were obtained by Yurinskii \cite{yurinskii1986averaging}. Dal Maso and Modica \cite{modica1986nonlinear} used average integration for the convenience of applying the subadditive ergodic theorem \cite{akcoglu1981ergodic} to achieve convergence. Recently, 
Armstrong, Bordas and Mourrat  \cite{armstrong2018quantitative} analyzed parabolic equations with random coefficients by using subadditive quantities derived from a variational interpretation of parabolic equations.  

  For general elliptic equations, qualitative homogenization results and quantitative results under a periodic setting have been obtained by  Xu \cite{XuQiang}. Nevertheless, existing works mainly focus on stochastic homogenization results only for equations without lower-order terms, leaving a significant knowledge gap in understanding the effects of lower-order terms on the homogenization process. 
Given these considerations, it becomes clear that a novel approach is required to address the homogenization of parabolic equations with lower-order terms. Therefore, we aim to address the problem:
	\begin{equation}\label{0.01}
		\partial_{t}p(x,t)=\nabla\cdot\left(\mathbf{a}(x,t)\nabla p(x,t)\right)+\mathbf{b}(x,t)\nabla p(x,t) +\mathbf{d}(x,t) p(x,t),
	\end{equation}
where $ \mathbf{a}(x,t) $, $ \mathbf{b}(x,t) $ and $ \mathbf{d}(x,t) $ are random coefficients  and  $p(x,t) $ is the solution of (\ref{0.01}). We are interested in the behavior of the solution on large scales. To emphasize the heterogeneity of the problem, the equation is rescaled with the parameter $0 < \epsilon\ll1$, {as follows}:
	\begin{equation}\label{0.2}
		\partial_{t}p^\epsilon(x,t)=\nabla\cdot\left(\mathbf{a}\left( \dfrac{x}{\epsilon},\dfrac{t}{\epsilon^2}\right)\nabla p^\epsilon(x,t)\right)+\mathbf{b}\left( \dfrac{x}{\epsilon},\dfrac{t}{\epsilon^2}\right)\nabla p^\epsilon(x,t) + \mathbf{d}\left( \dfrac{x}{\epsilon},\dfrac{t}{\epsilon^2}\right) p^\epsilon(x,t).
	\end{equation}

The aim of this paper is to study the asymptotic behavior of the solution $p^\epsilon$ of equation (\ref{0.2}) when the parameter $\epsilon$ goes to $0$. We find that the weak solution $p^\epsilon$ converges weakly in $H^1$ to the solution $p_0$ of the constant-coefficient equation:
\begin{equation*}
\partial_{t}p_{0}(x,t)=\nabla\cdot\left(\mathbf{\bar{a}}\nabla p_{0}(x,t)\right)+\mathbf{\bar{b}}\nabla p_{0}(x,t)+\mathbf{\bar{d}} p_{0}(x,t),
\end{equation*}
where $\mathbf{\bar{a}}$, $\mathbf{\bar{b}}$ and $\mathbf{\bar{d}}$, called homogenized coefficients, are deterministic. To achieve convergence, we employ a small trick by considering $\hat{p}^\epsilon=\exp{(-\Lambda t)}p^\epsilon$ instead of $p^\epsilon$.  Here, $\Lambda$ represents a positive constant governing the lower-order terms $\mathbf{b}$ and $\mathbf{d}$, with a detailed definition in Definition \ref{a b c}. It is evident that the convergence of $p^\epsilon$ can be directly derived from that of $\hat{p}^\epsilon$. Importantly, this approach applies exclusively to parabolic cases. When homogenizing elliptic equations,  results for equations with lower-order terms structured in this manner cannot be obtained. Additional conditions on the lower-order terms are necessary for the homogenization of elliptic equations.

This paper is organized as follows. In Section 2, we provide notation, definitions, and assumptions and present the main theorem. In Section 3, we mainly present generalized  {Meyers estimate}. In Section 4, we give the proof of the main theorem. 


\section{Preliminaries}

In this section, we will introduce some notation and definitions and the main result.
\subsection{Notation}
Throughout this work, we let the same letter $C$ denote positive constants that may vary from line to line.

A parabolic cylinder is any set of the form $ U\times I$ where $ U\in \mathbb{R}^d $ is a bounded Lipschitz domain and $ I=(I_{-},I_{+})\in \mathbb{R} $ is a bounded open interval.  Denote the parabolic boundary of $U\times I$ by $ \partial_\sqcup  (U\times I):=(\partial U\times I)\bigcup({U\times I_{-}}) $. For simplicity, set $V:=U\times I$. 
For a function $f: U\to \mathbb{R}$, its partial derivative is denoted by $\partial_{x_{i}}f$ and its gradient is represented as $\nabla f$. If $\mathbf{f}:U\to\mathbb{R}^{d} $ is a vector function and $ \mathbf{f}=(f_{1},...f_{d}) $, let $\nabla\cdot\mathbf{f}=\sum_{i=1}^{d}\partial_{x_{i}}f_{i}$ stand for the divergence of $\mathbf{f}$. The integral of $f$ is denoted as $\int_{U} 
		f:=\int_{U} 
		f(x)\mathrm{d}x.$
    
For $k\in[1,\infty]$, let $W^{1,k}(U)$ denote the Sobolev space. When $k=2$, write $H^1(U):=W^{1,2}(U)$. Let $W_{0}^{1,k}(U)$ be the closure of $C_{c}^{\infty}(U)$ in $W^{1,k}(U)$. Let $k^{\prime}:=k/(k-1)$. The dual space to $W^{1,k}(U)$ is denoted by ${W}^{-1,k^{\prime}}(U)$ and the norm is:
	\begin{equation*}
	\Vert f\Vert _{{W}^{-1,k^{\prime}}(U)}:=\sup\left\{
		\int_{U} 
		f g: g\in W_0^{1,k}(U),\Vert g\Vert _{W^{1,k}(U)}\leq1\right\}.
	\end{equation*}
  For simplicity, write $H^{-1}(U):=W^{-1,2}(U)$. For a Banach space $X$, and bounded Lipschitz domain $Y\in \mathbb{R}^n$($n\in \mathbb{N}$), $L^k(X;Y)$ denotes the set of measurable functions $f:Y\rightarrow X$ which satisfy:
    \begin{equation*}
    \Vert f\Vert _{L^k(X;Y)}
    =\left(\int_{Y}\Vert f\Vert^{k} _{X}\right)^{1/k}< \infty.
    \end{equation*}

The Sobolev space for parabolic equations is denoted by $W_{\mathrm{par}}^{1,k}(V)$, which is:
	\begin{equation*}
	W_{\mathrm{par}}^{1,k}(V):=\left\{f\in L^k(W^{1,k}(U); I):\partial_{t}f\in L^k(W^{-1,k}(U); I)\right\},
\end{equation*}
and the norm of $f$ is defined as
\begin{equation}\label{def w}
	\Vert f\Vert _{	W_{\mathrm{par}}^{1,k}(V)}:
	=\Vert f\Vert _{L^k(W^{1,k}(U); I)}+\Vert \partial_{t}f\Vert _{L^k(W^{-1,k}(U); I)}.
\end{equation}
The dual space to $W_{\mathrm{par}}^{1,k}(V)$  is represented as $\tilde{W}_{\mathrm{par}}^{-1,k^{\prime}}(V)$, with a norm defined as
	\begin{equation*}
	\Vert f\Vert _{\tilde{W}_{\mathrm{par}}^{-1,k^{\prime}}(V)}:=\sup\left\{
		\int_{V} 
		f g:g\in W_{\mathrm{par}}^{1,k}(V),\Vert g\Vert _{W_{\mathrm{par}}^{1,k}(V)}\leq1\right\}.
	\end{equation*}
The closure of $C_{c}^{\infty}(V)$ in $W_{\mathrm{par}}^{1,k}(V)$  is denoted by $W_{\mathrm{par},\sqcup}^{1,k}(V)$. The dual space to $W_{\mathrm{par},\sqcup}^{1,k}(V)$  is $W_{\mathrm{par}}^{-1,k^{\prime}}(V)$ .
Because $W_{\mathrm{par}}^{1,k}(V)\supseteq W_{\mathrm{par},_\sqcup}^{1,k}(V)$, we have $\Vert f\Vert _{{W}_{\mathrm{par}}^{-1,k^{\prime}}(V)}\leq\Vert f\Vert _{\tilde{W}_{\mathrm{par}}^{-1,k^{\prime}}(V)}$.

\subsection{Definitions and assumptions}

\begin{definition}\label{a b c}

	Fix constants $\lambda\in (1,\infty)$ and $\Lambda\in (0,\infty)$. Let $\Omega$ denote the set of $(\mathbf{a},\mathbf{b},\mathbf{d})$, where $\mathbf{a}(x,t)$ denote measurable mappings from $\mathbb{R}^d\times\mathbb{R}$ to $d$-by-$d$ symmetric matrices,  which satisfy the uniform ellipticity and boundedness condition that, for any $ \xi \in \mathbb{R}^d $,
	\begin{equation}\label{elli unifor}
		|\xi|^2\leq\xi\cdot \mathbf{a}(x,t)\xi\leq \lambda|\xi|^2,
	\end{equation}
	 $\mathbf{b}(x,t)$ denote mappings from $\mathbb{R}^d\times\mathbb{R}$ to the set of $1$-by-$d$ random vectors, which satisfy
	\begin{equation}\label{boundedness}
 \Vert \mathbf{b}\Vert^2_{L^{\infty}(\mathbb{R}^d\times\mathbb{R})}\leq \Lambda ,
	\end{equation}
and $\mathbf{d}(x,t)$ denote mappings from $\mathbb{R}^d\times\mathbb{R}$ to $\mathbb{R}$, which satisfy 
	\begin{equation}\label{boundedness1}
 \mathbf{d} \le 0\quad  and\quad \Vert \mathbf{d}\Vert^2_{L^{\infty}(\mathbb{R}^d\times\mathbb{R})}\leq \Lambda. 
	\end{equation} 
 \end{definition} 
 
	 For $V\subseteq \mathbb{R}^d\times\mathbb{R}$, let $\mathcal{F}_{V}$ denote the $\sigma$-field generated by the mappings
	\begin{equation*}
		(\mathbf{a},\mathbf{b},\mathbf{d})\mapsto\left( \int_{V}h(x,t)\mathbf{a}(x,t)\mathrm{d}x\mathrm{d}t,\int_{V}h(x,t)\mathbf{b}(x,t)\mathrm{d}x\mathrm{d}t,\int_{V}h(x,t)\mathbf{d}(x,t)\mathrm{d}x\mathrm{d}t\right) ,
	\end{equation*}
	where $h\in C^{\infty}_{c}(\mathbb{R}^d\times\mathbb{R})$. For simplicity, we write $ \mathcal{F} $ instead of $\mathcal{F}_{\mathbb{R}^d\times\mathbb{R}}$.  For each $y \in \mathbb{R}^d \times \mathbb{R}$, we set $ {T}_{y}: \Omega\to\Omega $ to be the shift, $ {T}_{y}(\mathbf{a}(\cdot),\mathbf{b}(\cdot),\mathbf{d}(\cdot))=\left( \mathbf{a}(\cdot+y),\mathbf{b}(\cdot+y),\mathbf{d}(\cdot+y)\right)  $.  Let $ \mathbb{P} $ denote a given probability measure on space $ (\Omega,\mathcal{F}) $, which satisfies the following:
 \begin{itemize}
		\item[\textbf{A1}.]  (Stationarity) For every $ y\in \mathbb{R}^d\times \mathbb{R}, \mathbb{P}\circ T_{y}= \mathbb{P}.	$
		
		\item[\textbf{A2}.]  (Ergodicity) 
		If $A\in \mathcal{F}$ and $T_{y}A=A$ for every $ y\in\mathbb{R}^d\times\mathbb{R}$, then $\mathbb{P}[A]\in{\{0,1\}}$.
	
	\end{itemize}	
 Let $\mathbb{E}[X]$ denote the expectation of an $\mathcal{F}$-measurable random variable $X$ with respect to $\mathbb{P}$.

\begin{definition}[First-order corrector]\label{first order}
 Let $\mathbf{a}$ be as in Definition \ref{a b c}. Let assumptions \textbf{A1} and \textbf{A2} hold.  For a unit vector $e\in \mathbb{R}^d$, $ \phi_{e} $ is a distributional solution to the corrector problem
\begin{equation}\label{corrector}
	\dfrac{\partial \phi_{e}}{\partial t}=\nabla\cdot(\mathbf{a}(e+\nabla\phi_{e})) \quad in\  \mathbb{R}^d\times\mathbb{R}.
\end{equation}
 The corrector $\phi_e $ satisfies the following properties  \cite[Definition 1]{Fischer2021} :
\begin{itemize}
		\item[1.]  The corrector $\phi_e$ has the regularity $\phi_e\in H^1_\mathrm{par,loc}(\mathbb{R}^d\times\mathbb{R})$, and satisfies $\int_{\square_{0}} \phi_e=0$ for $\mathbb{P}$-almost surely, where
  $$\Box_{0}:
	=\left(-\dfrac{1}{2},\dfrac{1}{2}\right)^{d}\times\left(-\dfrac{1}{2},\dfrac{1}{2}\right).$$

		\item[2.]  The gradient of the corrector $\nabla \phi_e$ is stationary with respect to $\mathbb{R}^d\times \mathbb{R} $-translation.
  
            \item[3.]  The gradient of the corrector $\nabla \phi_e$ satisfies
  \begin{equation}\label{corrector 3}
      \mathbb{E}\left[\nabla\phi_{e}\right]=0, \qquad \mathbb{E}\left[|\nabla\phi_{e}|^2\right]<\infty.
  \end{equation}
		\item[4.]  The corrector is sublinear, which means, for $\mathbb{P}$-almost surely, we have 
  \begin{equation}\label{sublinear}
      \lim_{r\to \infty}\frac{1}{r^2|Q_r|}\int_{Q_r}|\phi_e|^2 = 0,
  \end{equation}
	where $ Q_{r} $ denotes the parabolic cylinders: for $ r>0 $, set 
	\begin{equation}\label{Q_r}
		Q_{r}:=B_{r}\times I_{r},
	\end{equation}
where $ B_{r}:=\{x\in\mathbb{R}^{d}, |x|<r\} $ and $I_{r}:=(0,r^2]$. 
	\end{itemize}	 
\end{definition}

\begin{definition}[Homogenized coefficients]\label{homogenized a}
	We denote the homogenized coefficients by $\bar{\mathbf{a}}\in\mathbb{R}^{d\times d}$, $\bar{\mathbf{b}}\in\mathbb{R}^{1\times d}$, and $\bar{\mathbf{d}}\in\mathbb{R}$ and they  are defined by
	\begin{equation*}
		\left\{
		\begin{aligned}
			&	\bar{\mathbf{a}}e_{i}:=\mathbb{E}\left[\int_{\square_{0}} \mathbf{a}(e_{i}+\nabla\phi_{e_{i}})\right],	\\
			&
				\bar{\mathbf{b}}e_{i}:=\mathbb{E}\left[\int_{\square_{0}} \mathbf{b}(e_{i}+\nabla\phi_{e_{i}})\right],\\
			&
				\bar{\mathbf{d}}:=\mathbb{E}\left[\int_{\square_{0}} \mathbf{d}\right],
		\end{aligned}
		\right.
	\end{equation*}
 where $ e_{i} $ denotes the $ i $th basis vector of $\mathbb{R}^d $.
\end{definition}

\subsection{Main result}

\begin{theorem}\label{qualitative result}
 Let $\mathbf{a}$, $\mathbf{b}$, and $\mathbf{d}$ be as in Definition 2.1. We denote, for $0 < \epsilon\ll1$,
\begin{equation}\label{phi}
	\mathbf{a}^{\epsilon}(x,t):=\mathbf{a}\left( \dfrac{x}{\epsilon},\dfrac{t}{\epsilon^2}\right),\ \ 
	\mathbf{b}^{\epsilon}(x,t):=\mathbf{b}\left( \dfrac{x}{\epsilon},\dfrac{t}{\epsilon^2}\right),\ \ 
 \mathbf{d}^{\epsilon}(x,t):=\mathbf{d}\left( \dfrac{x}{\epsilon},\dfrac{t}{\epsilon^2}\right).
\end{equation}
Let assumptions \textbf{A1} and \textbf{A2} hold. Set a parabolic cylinder $V:=U\times I$, where $U\subseteq B_{1}:=\{x\in\mathbb{R}^{d}, |x|<1\} $ is a bounded Lipschitz domain, and $I\subseteq (0,1/4)$ is an open bounded interval. 
 Fix $f\in W^{1,2+\delta}_{\mathrm{par}}(V)$, $\delta>0$ and take $(p^\epsilon)_{\epsilon>0}$, $  p_{0} \in f+H^{1}_{\mathrm{par},\sqcup}(V) $ to be  weak  solutions of:
	\begin{equation}\label{q1}
		\begin{cases}
			\nabla\cdot(\mathbf{a}^\epsilon\nabla p^\epsilon)+\mathbf b^\epsilon \nabla p^\epsilon+\mathbf d^\epsilon p^\epsilon=\partial_{t}p^\epsilon&\quad in\  V,\\
			p^\epsilon=f&\quad on\  \partial_\sqcup V,
		\end{cases}
	\end{equation}
	and
	\begin{equation}\label{q2}
		\begin{cases}
			\nabla\cdot(\mathbf{\bar{a}}\nabla  p_{0})+\mathbf{\bar{b}}\nabla  p_{0}+\mathbf{\bar{d}} p_{0}=\partial_{t} p_{0}& \quad in\  V,\\
			p_{0}=f&\quad on\  \partial_\sqcup V,
		\end{cases}
	\end{equation}
	where $\bar{\mathbf{a}}$ , $\bar{\mathbf{b}} $ and $\bar{\mathbf{d}} $ are defined in Definition \ref{homogenized a}.
	 Then we have, as $\epsilon\to 0$, 
 \begin{equation}\label{quali result}
		\begin{cases}
			p^\epsilon \to p_{0}&\quad strongly\ in\  L^2(V),\\
			\nabla p^\epsilon \to \nabla p_{0}& \quad weakly\ in\  L^2(V).
		\end{cases}
	\end{equation}

\end{theorem}

\section{Prior estimates}
\subsection{Meyers inequality for general parabolic equation}

\begin{lemma}[Interior Caccioppoli inequality]\label{Interior Caccioppoli inequality}
	Let $Q_r$ be as in (\ref{Q_r}) and $\Lambda $ be as in Definition \ref{a b c}. Suppose that $p(x,t)\in H^{1}_{\mathrm{par}}(Q_{2r})$ and $h(x,t)\in L^2(H^{-1}(B_{2r}); I_{2r})$ satisfy 
	\begin{equation}\label{ca1}
		\partial_{t}p-\nabla\cdot(\mathbf{a}\nabla p)-\mathbf b\cdot\nabla p-\mathbf{d} p{+\Lambda p}=h\quad in\  Q_{2r}.
	\end{equation}
Then there exists a constant $C(d,\lambda)<\infty$ such that
	
	\begin{equation}\label{ca2}
		\Vert \nabla p\Vert _{L^{2}(Q_{r})}
		\leq C\left(r^{-1}\Vert p\Vert _{L^{2}(Q_{2r})}+\Vert h\Vert _{L^2(H^{-1}(B_{2r}); I_{2r})}\right),
	\end{equation}
	and
	
	\begin{equation}\label{rhs}
		\underset{s\in I_{r}}{\sup}\Vert p(x,s)\Vert _{L^{2}(B_{r})}
		\leq C\left(\Vert \nabla p\Vert _{L^{2}(Q_{2r})}+\Vert h\Vert _{L^2(H^{-1}(B_{2r}); I_{2r})}\right).
	\end{equation}
	
\end{lemma}
\begin{proof}
	
	Fix a cutoff function $\varphi\in C_{c}^{\infty}(Q_{2r})$ satisfying 
	\begin{equation}\label{ca condition}
		0\leq \varphi \le1,\ \varphi\equiv1\ \mathrm{in}\ Q_{r}, \ |\nabla \varphi|^2+|\partial_{t}\varphi|\leq Cr^{-2}.
	\end{equation}
	
With the test function $\varphi^2p\in L^2(H^{1}_0(B_{2r}); I_{2r})$, the weak formulation \cite[Chapter 8]{brezis2011functional} of (\ref{ca1}) is
	\begin{equation}\label{ca3}
		\int_{Q_{2r}}\nabla(\varphi^2p)\cdot \mathbf{a}\nabla p=\int_{Q_{2r}}\varphi^2p\cdot (h- \partial_{t}p+\mathbf b\nabla p+\mathbf{d} p{-\Lambda p}).
	\end{equation}
By using (\ref{elli unifor}) and (\ref{boundedness1}) and rearranging (\ref{ca3}), we have

	\begin{align}\label{in cacci}
		&\int_{Q_{2r}}\varphi^2(\nabla p)^2 \nonumber \\
  \le& \int_{Q_{2r}}(\varphi^2\nabla p)\cdot \mathbf{a}\nabla p \nonumber \\
  \leq& \int_{Q_{2r}}\varphi^2p\cdot (h- \partial_{t}p+\mathbf b\nabla p{-\Lambda p})-{2}\int_{Q_{2r}}\varphi p\nabla\varphi\cdot \mathbf a\nabla p \nonumber \\
  \leq& \int_{Q_{2r}}\varphi^2p\cdot (h- \partial_{t}p)+\left|\int_{Q_{2r}}\varphi^2p\cdot \mathbf b\nabla p\right|-\Lambda \Vert\varphi p\Vert_{L^{2}(Q_{2r})}^2\nonumber \\ 
    & +2 \left|\int_{Q_{2r}}\varphi p\nabla\varphi\cdot \mathbf a\nabla p\right|.
	\end{align}
 To deal with $\int_{Q_{2r}}\varphi^2p\cdot (h- \partial_{t}p)$, we adopt the approach outlined in the proof of \cite[Lemma B.3]{armstrong2018quantitative} to get
\begin{align}\label{in caccio1}
    & \int_{Q_{2r}}\varphi^2p\cdot (h- \partial_{t}p)\nonumber \\ 
     \le & -\int_{Q_{2r}} \partial_t\left(\frac{1}{2}\varphi^2p^2\right)+ \int_{Q_{2r}} \varphi |\partial_t \varphi| p^2 \nonumber \\ 
    & +C\Vert \varphi^2 p\Vert_{L^{2}(H^{1}(B_{2r}); I_{2r})}\Vert h\Vert _{L^2(H^{-1}(B_{2r}); I_{2r})} \nonumber \\
    = & -\frac{1}{2}\int_{B_{2r}} \varphi^2(4r^2,x) p^2(4r^2,x) + \int_{Q_{2r}} \varphi |\partial_t \varphi| p^2  \nonumber \\ 
    & + C\Vert \varphi^2 p\Vert_{L^{2}(H^{1}(B_{2r}); I_{2r})}\Vert h\Vert _{L^2(H^{-1}(B_{2r}); I_{2r})} \nonumber \\
  \leq& \int_{Q_{2r}} \varphi |\partial_t \varphi| p^2 + C\Vert \varphi^2 p\Vert_{L^{2}(H^{1}(B_{2r}); I_{2r})}\Vert h\Vert _{L^2(H^{-1}(B_{2r}); I_{2r})}.
\end{align}
By using the Poincar\'e inequality, (\ref{ca condition}), and Young's inequality, we have
\begin{align}\label{3.7}
     & \Vert \varphi^2 p\Vert_{L^{2}(H^{1}(B_{2r}); I_{2r})}\Vert h\Vert _{L^2(H^{-1}(B_{2r}); I_{2r})} \nonumber \\
  \leq& C \Vert \nabla(\varphi^2p)\Vert_{L^{2}(Q_{2r})}\Vert h\Vert _{L^2(H^{-1}(B_{2r}); I_{2r})}\nonumber \\
    \le & C (r^{-1}\Vert p\Vert_{L^{2}(Q_{2r})}+\Vert\varphi\nabla p\Vert_{L^{2}(Q_{2r})})\Vert h\Vert _{L^2(H^{-1}(B_{2r}); I_{2r})}\nonumber \\
    \le & C r^{-2} \Vert p\Vert_{L^{2}(Q_{2r})}^2 + \frac{1}{4}\Vert\varphi\nabla p\Vert_{L^{2}(Q_{2r})}^2+C\Vert h\Vert ^2_{L^2(H^{-1}(B_{2r}); I_{2r})}.
\end{align}
Combining (\ref{in caccio1}) with (\ref{3.7}) and by (\ref{ca condition}), we obtain
\begin{align}\label{3.777}
    & \int_{Q_{2r}}\varphi^2p\cdot (h- \partial_{t}p)\nonumber\\
  \leq& C r^{-2} \Vert p\Vert_{L^{2}(Q_{2r})}^2 + \frac{1}{4}\Vert\varphi\nabla p\Vert_{L^{2}(Q_{2r})}^2+C\Vert h\Vert ^2_{L^2(H^{-1}(B_{2r}); I_{2r})}.
\end{align}
According to Young's inequality and noting that $\Vert \mathbf{b}\Vert^2_{L^{\infty}(\mathbb{R}^d\times\mathbb{R})}\le \Lambda $, we see that
	\begin{align}\label{b lambda}
		&\left|\int_{Q_{2r}}\varphi^2p\cdot \mathbf b\nabla p\right|-\Lambda \Vert\varphi p\Vert_{L^{2}(Q_{2r})}^2 \nonumber \\ 
  \le& \Vert \mathbf{b}\Vert_{L^{\infty}(V)}\left|\int_{Q_{2r}}\varphi^2p\cdot \nabla p\right|-\Lambda \Vert\varphi p\Vert_{L^{2}(Q_{2r})}^2 \nonumber \\ 
  \leq&  \Lambda \Vert\varphi p\Vert_{L^{2}(Q_{2r})}^2+ \frac{1}{4} \Vert \varphi \nabla p\Vert_{L^{2}(Q_{2r})}^2-\Lambda \Vert\varphi p\Vert_{L^{2}(Q_{2r})}^2\nonumber \\ 
  \leq& \frac{1}{4} \Vert \varphi \nabla p\Vert_{L^{2}(Q_{2r})}^2.
	\end{align}
For the last term on the right-hand side of (\ref{in cacci}), by Young's inequality and (\ref{elli unifor}) and (\ref{ca condition}), we have
\begin{align}\label{in caccio}
		2 \left|\int_{Q_{2r}}\varphi p\nabla\varphi\cdot \mathbf a\nabla p\right|&\leq \frac{1}{4}\Vert\varphi\nabla p\Vert_{L^{2}(Q_{2r})}^2+ C \Vert (\nabla\varphi) p\Vert_{L^{2}(Q_{2r})}^2\nonumber\\
        &\le \frac{1}{4}\Vert\varphi\nabla p\Vert_{L^{2}(Q_{2r})}^2+ Cr^{-2} \Vert p\Vert_{L^{2}(Q_{2r})}^2,
\end{align}
Substituting (\ref{3.777}), (\ref{b lambda}) and (\ref{in caccio}) into (\ref{in cacci}), we conclude that
\begin{align}\label{ca01}
		\int_{Q_{2r}}\varphi^2(\nabla p)^2
  \leq C \left(r^{-2} \Vert p\Vert_{L^{2}(Q_{2r})}^2+\Vert h\Vert^2 _{L^2(H^{-1}(B_{2r}); I_{2r})}\right) .
	\end{align} 
 Using (\ref{ca condition}) and (\ref{ca01}), we have
 \begin{align}\label{ca04}
		\Vert \nabla p\Vert _{L^{2}(Q_{r})}\leq \Vert\varphi\nabla p\Vert_{L^{2}(Q_{2r})}
  \leq C \left(r^{-1}\Vert p\Vert _{L^{2}(Q_{2r})}+\Vert h\Vert _{L^{2}(H^{-1}(B_{2r}); I_{2r})}\right).
	\end{align} 
Therefore, we have proved (\ref{ca2}).

In the same way, we can get (\ref{rhs}). The only difference is that we test the equation (\ref{ca1}) with $\eta=\varphi^2p\mathbb{I}_{t<s}$ $(s\in I_{2r})$ and the weak formulation of (\ref{ca1}) becomes
	\begin{equation}\label{estab}
		\int_{Q_{2r}}\nabla\eta\cdot \mathbf a\nabla p=\int_{Q_{2r}}\eta(h-\partial_{t}p+\mathbf b\nabla p+\mathbf{d} p{-\Lambda p}).
	\end{equation}
As mentioned before, we get the upper bound of the right-hand side of (\ref{estab}) and the constraint for the left-hand side of (\ref{estab}). For more details, we refer the reader to \cite[ Lemma B.3]{armstrong2018quantitative}.
\end{proof}

Following Lemma \ref{Interior Caccioppoli inequality}, we obtain a global version of the Caccioppoli inequality.
\begin{lemma}[global Caccioppoli inequality]\label{global Caccioppoli inequality}
	Let $\Lambda$ and $Q_r$ be as in Lemma \ref{Interior Caccioppoli inequality} and let $V$ be as in Theorem \ref{qualitative result}. Let $H\in L^{2}(H^{-1}(U);I)$ and suppose that $ v\in H^{1}_{\mathrm{par},\sqcup}(V) $ is the unique weak solution to
\begin{equation}
		\begin{cases}
			\partial_{t}v-\nabla\cdot(\mathbf{a}\nabla v)-\mathbf{b}\nabla v-\mathbf{d} v{+\Lambda v}=H& \text{in $V$},\\
			v=0& \text{on $\partial_{\sqcup}V$}.
		\end{cases}
	\end{equation}
Then there exists a constant $C(V, d,\lambda)<\infty$ such that
	  \begin{align}\label{global cacci}
	\Vert\nabla v\Vert_{{L}^{2}(Q_{r}\cap V)}
		\le C\left(r^{-1}\Vert v\Vert _{L^{2}(Q_{2r}\cap V)}+\Vert H\Vert _{L^2(H^{-1}(B_{2r}\cap  U); I_{2r}\cap I)}\right).
	\end{align}

\end{lemma}

\begin{proof}
   The argument is omitted as it is a straightforward adaptation of Lemma \ref{Interior Caccioppoli inequality}. For more details, refer to \cite[Lemma B.6]{armstrong2018quantitative}.
\end{proof}

\begin{lemma}[Global Meyers estimate]\label{meyer}
Let $V$ be as in Theorem \ref{qualitative result}. Fix $l\ge2$ and suppose $f\in W^{1,l}_{\mathrm{par}}(V)$. Let $h\in L^{l}(W^{-1,l}(U);I)$ and suppose that $ p\in f+H^{1}_{\mathrm{par},\sqcup}(V) $ is the unique  weak solution to the Cauchy-Dirichlet problem:
	\begin{equation}\label{reverse}
		\begin{cases}
			\partial_{t}p-\nabla\cdot(\mathbf{a}\nabla p)-\mathbf{b}\nabla p-\mathbf{d} p{+\Lambda p}=h& \text{in $V$},\\
			p=f& \text{on $\partial_{\sqcup}V$}.
		\end{cases}
	\end{equation}
	Then there exist $\delta(V,d,\lambda)\in(0,l-2]$ and $C(V,d,\lambda)<\infty$ such that $p\in W^{1,2+\delta}_{\mathrm{par}}(V)$ and 	
	\begin{equation}\label{meyer1}
		\Vert\nabla p\Vert_{L^{2+\delta}(V)}\le C\left(\Vert f\Vert_{W^{1,2+\delta}_{\mathrm{par}}(V)}+\Vert h\Vert_{L^{2+\delta}\left(W^{-1,2+\delta}(U);I\right)}\right).
	\end{equation}

\end{lemma}

\begin{proof}
	We can assume that $f=0$ without loss of generality. Indeed, taking
	\begin{equation}\label{H}
		v=p-f,\  H=h+\nabla\cdot(\mathbf{a}\nabla f)+\mathbf{b}\nabla f+\mathbf{d} f{-\Lambda f}-\partial_{t}f,
	\end{equation}
we see that $v$ is a solution to
\begin{equation}\label{reversee}
		\begin{cases}
			\partial_{t}v-\nabla\cdot(\mathbf{a}\nabla v)-\mathbf{b}\nabla v-\mathbf{d} v{+\Lambda v}=H& \text{in $V$},\\
			v=0& \text{on $\partial_{\sqcup}V$}.
		\end{cases}
	\end{equation}
	
By Lemma \ref{Interior Caccioppoli inequality}, we can get the reverse H$\mathrm{\ddot{o}}$lder inequality \cite[Lemma B.7]{armstrong2018quantitative} for (\ref{reversee}), which states that for $1<m<2$ and $k>0$, we have
	\begin{align}\label{reverse holder}
		&\Vert\nabla v\Vert_{{L}^{2}(Q_{r}\cap V)}\nonumber\\
  \le&\dfrac{C}{k}\Vert\nabla v\Vert_{{L}^{m}(Q_{4r}\cap V)}+k\Vert\nabla v\Vert_{{L}^{2}(Q_{4r})}+C\Vert H\Vert_{{L}^{2}({H}^{-1}(B_{4r}\cap U); I_{4r}\cap I)}.
	\end{align}
By (\ref{reverse holder}) and a Gehring-type lemma \cite[ Proposition 5.1]{modica1979regularity}, we have
	\begin{equation}
		\Vert\nabla v\Vert_{{L}^{2+\delta}(Q_{r}\cap V)}\le C\left(\Vert\nabla v\Vert_{{L}^{2}(Q_{4r}\cap V)}
		+\Vert H\Vert_{{L}^{2+\delta}( W^{-1,2+\delta}(B_{4r}\cap U); I_{4r}\cap I)}\right).
	\end{equation}
Let $r$ be large enough, then
	\begin{equation}\label{lemma}
		\Vert\nabla v\Vert_{{L}^{2+\delta}( V)}\le C\left(\Vert\nabla v\Vert_{{L}^{2}( V)}
		+\Vert H\Vert_{{L}^{2+\delta}( W^{-1,2+\delta}( U); I)}\right).
	\end{equation}
By using (\ref{global cacci}) and letting $r$ be large enough, we get
  \begin{align}\label{me11}
	\Vert\nabla v\Vert_{{L}^{2}( V)}
		\le C\Vert H\Vert_{{L}^{2}(H^{-1}( U); I)}.
	\end{align}
    Note that $p=v+f$ and combine (\ref{lemma}) with (\ref{me11}), then we have
	\begin{align}\label{me111}
	\Vert\nabla p\Vert_{{L}^{2+\delta}(V)}&\le \Vert\nabla v\Vert_{{L}^{2+\delta}(V)}+\Vert\nabla f\Vert_{{L}^{2+\delta}(V)}\nonumber\\
		&\le C\left(\Vert H\Vert_{{L}^{2}( H^{-1}(U); I)}+\Vert H\Vert_{{L}^{2+\delta}( W^{-1,2+\delta}( U); I)}\right)+\Vert\nabla f\Vert_{{L}^{2+\delta}( V)}\nonumber\\
		&\le C\Vert H\Vert_{{L}^{2+\delta}( W^{-1,2+\delta}( U); I)}+\Vert\nabla f\Vert_{{L}^{2+\delta}( V)}.
	\end{align}
Applying (\ref{H}) and (\ref{def w}), we obtain
\begin{align}\label{me2}
	&\Vert H\Vert_{{L}^{2+\delta}( W^{-1,2+\delta}( U);I)}+\Vert\nabla f\Vert_{{L}^{2+\delta}( V)}\nonumber\\
		\le &C(\|h\|_{L^{2+\delta}\left(W^{-1,2+\delta}(U); I\right)}+\|f\|_{L^{2+\delta}\left(W^{1,2+\delta}(U); I\right)})\nonumber\\
  &+C(\|\partial_{t}f\|_{L^{2+\delta}\left(W^{-1,2+\delta}(U); I\right)}+\Vert\nabla f\Vert_{{L}^{2+\delta}( V)} )\nonumber\\
		\le &C(\Vert h\Vert_{L^{2+\delta}\left(W^{-1,2+\delta}(U); I\right)}+\Vert f\Vert_{W^{1,2+\delta}_{\mathrm{par}}(V)} ).
	\end{align}
Combining (\ref{me111}) with (\ref{me2}) yields (\ref{meyer1}).
 \end{proof}

\subsection{Two-scale expansion}
 The two-scale expansion is one of the most crucial techniques for studying homogenization, and more details on this method can be found in the work of Bensoussan, Lions, and Papanicolaou  \cite{bensoussan2011asymptotic}. In this study, we utilize this method to conduct homogenization for the random case with lower-order terms. We start with a hypothesis: the solution $ p^\epsilon $  of (\ref{q1}) is a series  {as follows}:
	\begin{align}\label{expansion}
		p^\epsilon(x,t)=\sum_{i=0}^{\infty}\epsilon^{i}p_{i}\left(x,\dfrac{x}{\epsilon},t,\dfrac{t}{\epsilon^{2}}\right).
	\end{align}
Substituting (\ref{expansion}) to (\ref{q1}) and setting $ y=x/\epsilon$ and $ s=t/\epsilon^{2} $, we obtain
\begin{align*}
	0=&\nabla\cdot(\mathbf{a}^\epsilon\nabla p^\epsilon)+\mathbf b^\epsilon \nabla p^\epsilon+\mathbf{d}^\epsilon p^\epsilon-\partial_{t}p^\epsilon\\
		=&\dfrac{1}{\epsilon^2}\nabla_{y}\cdot\left(\mathbf{a}^\epsilon\nabla_{y}p_{0}\right)
		+\dfrac{1}{\epsilon}\nabla_{x}\cdot\left(\mathbf{a}^\epsilon\nabla_{y}p_{0}\right)+\dfrac{1}{\epsilon}\nabla_{y}\cdot\left(\mathbf{a}^\epsilon\left(\nabla_{x}p_{0}+\nabla_{y}p_{1}\right)\right)\\
		&+\nabla_{x}\cdot\left(\mathbf{a}^\epsilon\left(\nabla_{x}p_{0}+\nabla_{y}p_{1}\right)\right)+\dfrac{1}{\epsilon} \mathbf{b}^\epsilon\nabla_{y}p_{0}
		+\mathbf{b}^\epsilon\nabla_{x}p_{0}+\mathbf{b}^\epsilon\nabla_{y}p_{1}\\
		&+\mathbf{d}^\epsilon p_{0} -
		\partial_{t}p_{0}
		-\dfrac{1}{\epsilon^{2}}\partial_{s}p_{0}
		-\dfrac{1}{\epsilon}\partial_{s}p_{1}
		-\partial_{s}p_{2}-\cdot\cdot\cdot.
\end{align*}
By determining powers of $\epsilon$ in the expansion and letting $\epsilon \to 0$, we get $p_{0}\left(x,y,t,s\right)=p_{0}(x,t)$ and
\begin{align}\label{not dep}
	\nabla_{y}\cdot\left(\mathbf{a}^\epsilon\nabla_{y}p_{1}\right)-\partial_{s}p_{1}=-\sum_{i=1}^{d}\nabla_{y}\mathbf{a}^\epsilon e_{i}\partial_{x_{i}}p_{0}(x,t).
\end{align}
Combining (\ref{not dep}) with (\ref{corrector}), we get
\begin{align*}
	-\partial_{s}p_{1}+\nabla_{y}\cdot\left(\mathbf{a}^\epsilon\nabla_{y}p_{1}\right)
	=\sum_{i=1}^{d}\nabla_{y}\left(\mathbf{a}^\epsilon\nabla_{y}\phi_{e_{i}}(y,s)\right)\partial_{x_{i}}p_{0}(x,t)-\sum_{i=1}^{d}\partial_{s}\phi_{e_{i}}(y,s)\partial_{x_{i}}p_{0}(x,t).
\end{align*}
Therefore, we can identify $ p_{1} $ by $\phi_{e}$ and $ p_{0} $, that is
\begin{equation}\label{p1}
	p_{1}=\sum_{i=1}^{d} \partial_{x_{i}} p_{0} (x,t)\phi_{e_{i}}(y,s).
\end{equation}
  {For simplicity}, we denote, for $0 < \epsilon\ll1$,
\begin{equation}
	\phi_e^{\epsilon}(x,t):=\phi_e\left( \dfrac{x}{\epsilon},\dfrac{t}{\epsilon^2}\right)=\phi_{e}(y,s).
\end{equation}
In our setting, we demonstrate that the first-order corrector $\phi$ does not depend on the lower-order terms, thereby supporting the definition of the first-order corrector as given in (\ref{corrector}). However, it is essential to acknowledge that the definition of the first-order corrector may vary depending on the specific form of the rescaled lower-order terms.


\subsection{Convergence estimates}
\begin{lemma} \label{convergence L to H}\cite[Exercise 1.9]{armstrong2019quantitative}
 Assume that the sequence $ \{f_{m}\}_{m\in\mathbb{N}}\subseteq L^{2}(V) $ is uniformly bounded. Then the following statements are equivalent:
\begin{enumerate}
	\item $f_{m}\to f$ weakly in $ L^{2}(V) $, as $m\to \infty$;
	\item $\Vert f_{m}-f\Vert_{{\tilde{H}}_{\mathrm{par}}^{-1}(V)}\to0$, as $m\to \infty$.
\end{enumerate}

\end{lemma}
\begin{proof}
(1) $\Rightarrow$ (2).
Let us begin with the definition of ${\tilde{H}}_{\mathrm{par}}^{-1}$ norm and we get

\begin{align}\label{exer}
	\Vert f_{m}-f\Vert _{{\tilde{H}}_{\mathrm{par}}^{-1}(V)}
	&=\sup\left\{
	\int_{V} 
	(f_{m}-f)g:g\in H_{\mathrm{par}}^{1}(V),\Vert g\Vert _{{H}_{\mathrm{par}}^{1}(V)}\leq1\right\}\nonumber\\
	&= \max \left\{
	\int_{V} 
	(f_{m}-f)g:g\in H_{\mathrm{par}}^{1}(V),\Vert g\Vert _{{H}_{\mathrm{par}}^{1}(V)}\leq1\right\}\nonumber\\
        &=\int_{V} (f_{m}-f)g_{m}.
\end{align}
Here $ g_m $ is an element that attains the maximum in the second line.   
Using Rellich's compactness theorem \cite[Theorem 5]{Evans1990}, we have a convergent subsequence in $L^2(V)$, denoted by the same symbol $\{g_m \}$, with a limit $ g $ in $L^2 (V)$. 
Taking $m \to \infty $ in (\ref{exer}), we obtain (2).

(2) $\Rightarrow$ (1).
For all $ g\in L^{2}(V)$ and a function $h\in H^{1}_{\mathrm{par}}(V)$ satisfying $\Vert h\Vert_{{H}_{\mathrm{par}}^{1}(V)}\leq {C} $, by the Cauchy-Schwarz inequality, 
we have
\begin{align}\label{f-f}
	\left|
	\int_{V} 
	(f_{m}-f)g\right|
 \leq&\left|
	\int_{V} 
	(f_{m}-f)(g-h)\right|+\left|
	\int_{V} 
	(f_{m}-f)h\right|\nonumber\\
	\leq& \Vert f_{m}-f\Vert _{{{L}}^{2}(V)}\Vert g-h\Vert _{{{L}}^{2}(V)}+C\Vert f_{m}-f\Vert _{{\tilde{H}}_{\mathrm{par}}^{-1}(V)}\nonumber\\
	\leq& \underset{m\in\mathbb{N}}{\sup}\left(\Vert f_{m}\Vert _{{{L}}^{2}(V)}+\Vert f\Vert_{{{L}}^{2}(V)}\right)\underset{\Vert h\Vert _{{H}_{\mathrm{par}}^{1}(V)}\leq {C}}{\inf}\Vert g-h\Vert _{{{L}}^{2}(V)}\nonumber\\
 &+ C\Vert f_{m}-f\Vert _{{\tilde{H}}_{\mathrm{par}}^{-1}(V)}.
\end{align}
By selecting a proper $C$ such that the right-hand side of (\ref{f-f}) is small enough and using (2), we get (1).
\end{proof}

\begin{proposition}\label{convergence a,b}

	The homogenized coefficients $\bar{\mathbf{a}}$, $\bar{\mathbf{b}}$, and $\bar{\mathbf{d}}$ are defined in Definition \ref{homogenized a} and the first-order corrector $\phi_{e}$ is defined in Definition \ref{first order}. Note that $\epsilon\nabla\phi_{e_{i}}^{\epsilon}=\nabla \phi_{e_{i}}\left( x/{\epsilon},t/{\epsilon^{2}} \right)$. We have the following convergence properties, as $\epsilon\to 0$:
  \begin{equation}\label{weak conv}
		\begin{cases}
           \epsilon\nabla\phi_{e_{i}}^{\epsilon} \to 0&\quad weakly\ in\ L^2(\square_{0}),\\
           \mathbf{a}^\epsilon \left( e_{i}+\epsilon\nabla\phi_{e_{i}}^{\epsilon}\right) \to\bar{\mathbf{a}}e_{i}& \quad weakly\ in\ L^2(\square_{0}),\\	
			\mathbf{b}^\epsilon \left( e_{i}+\epsilon\nabla\phi_{e_{i}}^{\epsilon}\right) \to\bar{\mathbf{b}}e_{i}& \quad weakly\ in\ L^2(\square_{0}),\\
			\mathbf{d}^\epsilon \to\bar{\mathbf{d}}& \quad weakly\ in\ L^2(\square_{0}).
		\end{cases}
	\end{equation}
 
\end{proposition}

\begin{proof}
By (\ref{corrector 3}) and the ergodic theorem \cite[Lemma 2.23]{Stefan2018}, we have, for all $ g(x,t)\in L^{2}(\square_{0})$, 
\begin{equation*}
\int_{\square_{0}} \nabla\phi_{e_{i}}\left( \dfrac{x}{\epsilon},\dfrac{t}{\epsilon^{2}} \right)g(x,t)\mathrm{d}x\mathrm{d}t\to \mathbb{E}\left[ \nabla\phi_{e_{i}}\left( \dfrac{x}{\epsilon},\dfrac{t}{\epsilon^{2}} \right) \right]\int_{\square_{0}} g(x,t)\mathrm{d}x\mathrm{d}t=0,\quad \mathrm{as}\ \epsilon \to 0.
\end{equation*}
By using (\ref{elli unifor}), (\ref{boundedness}) and (\ref{corrector 3}), we can use the ergodic theorem again to get, for all $ g(x,t)\in L^{2}(\square_{0})$, as $\epsilon \to 0$,
\begin{equation*}
		\left\{
		\begin{aligned}
			&  \int_{\square_{0}} \mathbf{a}^\epsilon \left( e_{i}+\epsilon\nabla\phi_{e_{i}}^{\epsilon}\right)g(x,t)\mathrm{d}x\mathrm{d}t\to \mathbb{E}\left[\int_{\square_{0}} \mathbf{a}^\epsilon \left( e_{i}+\epsilon\nabla\phi_{e_{i}}^{\epsilon}\right)\right]\int_{\square_{0}} g(x,t)\mathrm{d}x\mathrm{d}t,\\
			&
			 \int_{\square_{0}} \mathbf{b}^\epsilon \left( e_{i}+\epsilon\nabla\phi_{e_{i}}^{\epsilon}\right) g(x,t)\mathrm{d}x\mathrm{d}t\to \mathbb{E}\left[\int_{\square_{0}} \mathbf{b}^\epsilon \left( e_{i}+\epsilon\nabla\phi_{e_{i}}^{\epsilon}\right) \right]\int_{\square_{0}} g(x,t)\mathrm{d}x\mathrm{d}t,\\
			&
			 \int_{\square_{0}} \mathbf{d}^\epsilon g(x,t)\mathrm{d}x\mathrm{d}t\to \mathbb{E}\left[\int_{\square_{0}} \mathbf{d}^\epsilon \right]\int_{\square_{0}} g(x,t)\mathrm{d}x\mathrm{d}t.
		\end{aligned}
		\right.
	\end{equation*}
 Recalling the Definition \ref{homogenized a}, we have proved (\ref{weak conv}).
\end{proof}


\section{Proof for qualitative homogenization}

\begin{theorem}\label{quali0}
 Let $\Lambda$ be as in Lemma \ref{Interior Caccioppoli inequality} and let $V$ be as in Theorem \ref{qualitative result}.  Fix $f\in W^{1,2+\delta}_{\mathrm{par}}(V)$ with $\delta>0$. Let  $\hat{p}^\epsilon:=\exp{(-\Lambda t)}p^\epsilon $ and $\hat{p}_0:=\exp{(-\Lambda t)}p_0 $  be weak solutions of 
	\begin{equation}\label{hat q1}
		\begin{cases}
			\nabla\cdot(\mathbf{a}^\epsilon\nabla \hat{p}^\epsilon)+\mathbf b^\epsilon \nabla \hat{p}^\epsilon +\mathbf d^\epsilon \hat{p}^\epsilon - \Lambda \hat{p}^\epsilon=\partial_{t}\hat{p}^\epsilon&\quad in\  V,\\
			\hat{p}^\epsilon=f& \quad on\  \partial_\sqcup V,
		\end{cases}
	\end{equation}
	and
	\begin{equation}\label{hat q2}
		\begin{cases}
			\nabla\cdot(\mathbf{\bar{a}}\nabla  \hat{p}_0)+\mathbf{\bar{b}}\nabla  \hat{p}_0+\mathbf{\bar{d}} \hat{p}_0-\Lambda \hat{p}_0=\partial_{t} \hat{p}_0& \quad in\  V,\\
			\hat{p}_0=f&\quad on\  \partial_\sqcup V,
		\end{cases}
	\end{equation}
 respectively.
 Then there exist an exponent $\beta(\delta,d,\lambda)>0$ and a constant $C(\delta,d,\lambda,V)>0$ such that for every $r\in(0,1)$, we have
	\begin{align}\label{result 1}
		\Vert \hat{p}^\epsilon- \hat{p}_0\Vert_{L^{2}(V)}+\Vert\nabla \hat{p}^\epsilon-\nabla \hat{p}_0\Vert_{{\tilde{H}}^{-1}_{\mathrm{par}}(V)}
		\le C\Vert f\Vert_{W^{1,2+\delta}_{\mathrm{par}}(V)}\left( r^{\beta}+\dfrac{1}{r^{4+d/2}}\mathcal{E}(\epsilon)\right),
	\end{align}
where ${\tilde{H}}^{-1}_{\mathrm{par}}(V)={\tilde{W}}^{-1,2}_{\mathrm{par}}(V)$ and
\begin{align}\label{da E}
	\mathcal{E}(\epsilon):=\sum_{i=1}^{d} & \left( \epsilon\left\lVert\phi_{e_{i}}^{\epsilon} \right\lVert_{{L}^{2}(\square_{0})}+\left\lVert\epsilon\nabla\phi_{e_{i}}^{\epsilon} \right\lVert_{{\tilde{H}}_{\mathrm{par}}^{-1}(\square_{0})}+\left\lVert\mathbf{a}^{\epsilon}\left(e_{i}+\epsilon\nabla\phi_{e_{i}}^{\epsilon}\right)-\bar{\mathbf{a}}e_{i} \right\lVert_{{\tilde{H}}_{\mathrm{par}}^{-1}(\square_{0})}\right.\nonumber\\
		&+\left.\left\lVert\mathbf{b}^{\epsilon}\left(e_{i}+\epsilon\nabla\phi_{e_{i}}^{\epsilon}\right)-\bar{\mathbf{b}}e_{i} \right\lVert_{{\tilde{H}}_{\mathrm{par}}^{-1}(\square_{0})}+\left\lVert\mathbf{d}^{\epsilon}-\bar{\mathbf{d}}\right\lVert_{{\tilde{H}}_{\mathrm{par}}^{-1}(\square_{0})}\right).
\end{align}
\end{theorem}

Let $\hat{p}_0$ be as in (\ref{hat q2}). From (\ref{p1}), we consider the test function defined as
\begin{equation}\label{test function}
	w^{\epsilon}
	:=\hat{p}_0+\eta_{r}\epsilon\sum_{i=1}^{d}(\partial_{x_{i}} \hat{p}_0)\phi_{e_{i}}^\epsilon,
\end{equation}
where $ \eta_{r} $ is a smooth cutoff function satisfying the conditions,
	\begin{equation}\label{eta}
		\left\{
		\begin{aligned}
			&  0\le\eta_{r}\le1,\quad\eta_{r}\equiv1\ \mathrm{in}\    U_{2r}\times \tilde{I}_{2r},\\
			&
		       \eta_{r}\equiv0\ \mathrm{in}\ (U\times I)\setminus(U_{r}\times \tilde{I}_{r}),\\
	        &	\forall k,l\in\mathbb{N},\ \left|\nabla^{k}\partial_{t}^{l}\eta_{r}\right|\le C_{k+2l}r^{-(k+2l)},
		\end{aligned}
		\right.
	\end{equation}
	where $U_{r}:=\{x\in U: \mathrm{dist}(x,\partial U)>r\} $ with $r\in(0,1)$, and $ \tilde{I}_{r}:=(I_{-}+r^2,I_{+}) $.
Here, we will also use the standard pointwise estimate \cite[Chapter 2]{evans2010partial} which gives 
\begin{equation}\label{point}
	\left\lVert\nabla^{k}\partial_{t}^{l} \hat{p}_0\right\rVert_{L^{\infty}(U_{r}\times I_{r})}\le C_{k+2l}r^{-(k+2l)-(2+d)/2}\Vert\nabla f\Vert_{L^{2}(V)}.
\end{equation}

\begin{lemma}
Let $\hat{p}^\epsilon$ and $w^{\epsilon}$ be as in (\ref{hat q1}) and (\ref{test function}). We have
\begin{align}\label{comebine with step1}
		\Vert \hat{p}^\epsilon- w^{\epsilon}\Vert_{L^{2}(H^{1}(U); I)}\le C\Vert \left( \partial_{t}-\nabla\cdot\mathbf{a}^{\epsilon}\nabla-\mathbf{b}^{\epsilon}\nabla-\mathbf{d}^{\epsilon}+\Lambda \right)  w^{\epsilon}\Vert_{L^{2}(H^{-1}(U); I)}.
	\end{align}    
\end{lemma}
\begin{proof}
     
Using (\ref{elli unifor}), we have 
	\begin{align}\label{zui 1}
		\Vert \nabla(\hat{p}^\epsilon- w^{\epsilon})\Vert^{2}_{L^{2}(V)}\le \int_{V}\nabla(\hat{p}^\epsilon- w^{\epsilon})\cdot\mathbf{a}^{\epsilon}\nabla(\hat{p}^\epsilon- w^{\epsilon}).
	\end{align}
 Since $ \hat{p}^\epsilon- w^{\epsilon} \in H^{1}_{\mathrm{par},_\sqcup}(V)$, by using integration by parts, we have
\begin{align}\label{zui2}
\int_{V}\nabla(\hat{p}^\epsilon- w^{\epsilon})\cdot\mathbf{a}^{\epsilon}\nabla(\hat{p}^\epsilon- w^{\epsilon})
		= \int_{V}(\hat{p}^\epsilon- w^{\epsilon})(-\nabla\cdot\mathbf{a}^{\epsilon}\nabla(\hat{p}^\epsilon- w^{\epsilon})).
	\end{align}
Recalling that $ \partial_\sqcup  V:=(\partial U\times I)\bigcup({U\times I_{-}}) $ and $ \hat{p}^\epsilon- w^{\epsilon} \in H^{1}_{\mathrm{par},_\sqcup}(V)$, we get 
	\begin{align}\label{boundary t}
		\int_{U}\int_{I}(\hat{p}^\epsilon- w^{\epsilon})\partial_{t}(\hat{p}^\epsilon- w^{\epsilon})
		&=\frac{1}{2}\Vert (\hat{p}^\epsilon- w^{\epsilon})(x,I_{+})\Vert^{2}_{L^{2}(U)}-\frac{1}{2}\Vert (\hat{p}^\epsilon- w^{\epsilon})(x,I_{-})\Vert^{2}_{L^{2}(U)}\nonumber\\
		&=\frac{1}{2}\Vert (\hat{p}^\epsilon- w^{\epsilon})(x,I_{+})\Vert^{2}_{L^{2}(U)}\ge 0.
	\end{align}
 By using (\ref{boundedness}), and Young's inequality, and noting that $\Vert \mathbf{b}\Vert_{L^{\infty}(\mathbb{R}^d\times\mathbb{R})}\le \sqrt{\Lambda} $, we have
	\begin{align}\label{b part}
		\left|\int_{V}(\hat{p}^\epsilon- w^{\epsilon})\left(-\mathbf{b}^{\epsilon}\nabla\right) (\hat{p}^\epsilon- w^{\epsilon})\right|
            &\le \Vert \mathbf{b}\Vert_{L^{\infty}(V)}\int_{I}\Vert \nabla(\hat{p}^\epsilon- w^{\epsilon})\Vert_{L^{2}(U)}\Vert(\hat{p}^\epsilon- w^{\epsilon})\Vert_{L^{2}(U)}\nonumber\\
	            &\le \frac{1}{4} \Vert \nabla(\hat{p}^\epsilon- w^{\epsilon}) \Vert^2_{L^{2}(V)}+ \Lambda \Vert \hat{p}^\epsilon- w^{\epsilon} \Vert^2_{L^{2}(V)}.
	\end{align}
 By using (\ref{boundedness1}), we have
 \begin{align}\label{d part}
		0\le\int_{V}(\hat{p}^\epsilon- w^{\epsilon})\left(-\mathbf{d}^{\epsilon}\right) (\hat{p}^\epsilon- w^{\epsilon}).
	\end{align}
From (\ref{boundary t})-(\ref{d part}), we get
 	\begin{align} \label{Lam}
		0\le \int_{V}(\hat{p}^\epsilon- w^{\epsilon})\left(\partial_{t}-\mathbf{b}^{\epsilon}\nabla -\mathbf{d}^{\epsilon}+ \Lambda\right) (\hat{p}^\epsilon- w^{\epsilon})+\frac{1}{4} \Vert \nabla(\hat{p}^\epsilon- w^{\epsilon}) \Vert^2_{L^{2}(V)}.
	\end{align}
Adding $\frac{3}{4} \Vert \nabla(\hat{p}^\epsilon- w^{\epsilon}) \Vert^2_{L^{2}(V)}$ to both sides of (\ref{Lam}), combining with (\ref{zui 1}) and (\ref{zui2}), and from ($\ref{hat q1}$), we obtain
 \begin{align}\label{zui3}
\frac{3}{4} \Vert \nabla(\hat{p}^\epsilon- w^{\epsilon}) \Vert^2_{L^{2}(V)}
&\le \int_{V}(\hat{p}^\epsilon- w^{\epsilon})\left(\partial_{t}-\mathbf{b}^{\epsilon}\nabla -\mathbf{d}^{\epsilon}+ \Lambda\right) (\hat{p}^\epsilon- w^{\epsilon})+ \Vert \nabla(\hat{p}^\epsilon- w^{\epsilon}) \Vert^2_{L^{2}(V)}\nonumber \\ 
  &\le \int_{V}(\hat{p}^\epsilon- w^{\epsilon})\left( \partial_{t}-\nabla\cdot\mathbf{a}^{\epsilon}\nabla-\mathbf{b}^{\epsilon}\nabla-\mathbf{d}^{\epsilon}+\Lambda\right)  (-w^{\epsilon}).
	\end{align}
By using the Poincar\'e inequality, we have
	\begin{equation}\label{zuichu}
		\Vert \hat{p}^\epsilon- w^{\epsilon}\Vert^{2}_{L^{2}(H^{1}(U); I)}\le 	C\Vert \nabla(\hat{p}^\epsilon- w^{\epsilon})\Vert^{2}_{L^{2}(V)}.
	\end{equation}
 Combining (\ref{zui3}) with (\ref{zuichu}) and using the Cauchy–Schwarz inequality, we obtain
	\begin{align}\label{zui1}
		&\Vert \hat{p}^\epsilon- w^{\epsilon}\Vert^{2}_{L^{2}(H^{1}(U); I)}\nonumber\\
		\le& C\int_{V}(\hat{p}^\epsilon- w^{\epsilon})\left( \partial_{t}-\nabla\cdot\mathbf{a}^{\epsilon}\nabla-\mathbf{b}^{\epsilon}\nabla-\mathbf{d}^{\epsilon}+ \Lambda\right) ( -w^{\epsilon})\nonumber\\
		\le& C\Vert \hat{p}^\epsilon- w^{\epsilon}\Vert_{L^{2}(H^{1}(U); I)}\Vert \left( \partial_{t}-\nabla\cdot\mathbf{a}^{\epsilon}\nabla-\mathbf{b}^{\epsilon}\nabla-\mathbf{d}^{\epsilon}+\Lambda\right)  w^{\epsilon}\Vert_{L^{2}(H^{-1}(U); I)}.
	\end{align}
 Dividing both sides of (\ref{zui1}) by $ \Vert \hat{p}^\epsilon- w^{\epsilon}\Vert_{L^{2}(H^{1}(U); I)} $, we finally have proved (\ref{comebine with step1}).
\end{proof}

\begin{lemma}
Let $\mathcal{E}(\epsilon)$ be as in (\ref{da E}). Then there exists $C(\delta,d,\lambda)>0$ such that
\begin{align}\label{step 1}
		&\Vert(\partial_{t}-\nabla\cdot\mathbf{a}^{\epsilon}\nabla-\mathbf b^\epsilon\nabla-\mathbf{d}^{\epsilon}+ \Lambda)w^{\epsilon}\Vert_{L^{2}(H^{-1}(U); I)}\nonumber\\
  \le& C\left(r^{\frac{\delta}{4+2\delta}}+r^{4+d/2}\mathcal{E}(\epsilon)^{\frac{1}{2}}\right)\Vert f\Vert_{W_{\mathrm{par}}^{1,2+\delta}(V)}.
	\end{align}

\end{lemma}
\begin{proof}
By substituting the definition (\ref{test function}) of $ w^{\epsilon} $ into (\ref{hat q1}) for $ \hat{p}^\epsilon $ and by (\ref{corrector}), we get
 \begin{align}\label{w shizi}
		&(\partial_{t}-\nabla\cdot\mathbf{a}^{\epsilon}\nabla-\mathbf b^\epsilon\nabla-\mathbf{d}^{\epsilon}+\Lambda)w^{\epsilon}\nonumber\\
		=&\partial_{t} \hat{p}_0+\sum_{i=1}^{d}\epsilon\phi_{e_{i}}^{\epsilon}\partial_{t}(\eta_{r}\partial_{x_{i}} \hat{p}_0)-\sum_{i=1}^{d}\nabla\left(\eta_{r}\partial_{x_{i}} \hat{p}_0\right)\mathbf{a}^{\epsilon}\left(e_{i}+\epsilon\nabla\phi_{e_{i}}^{\epsilon}\right)\nonumber\\
		&-\nabla\cdot\left(\mathbf{a}^{\epsilon}\left(\sum_{i=1}^{d}\epsilon\phi_{e_{i}}^{\epsilon}\nabla(\eta_{r}\partial_{x_{i}} \hat{p}_0)+(1-\eta_{r})\nabla  \hat{p}_0\right)\right)\nonumber\\
		&-\mathbf{b}^{\epsilon}\left(\sum_{i=1}^{d}\epsilon\phi_{e_{i}}^{\epsilon}\nabla(\eta_{r}\partial_{x_{i}} \hat{p}_0)+(1-\eta_{r})\nabla  \hat{p}_0+\sum_{i=1}^{d}\eta_{r}\partial_{x_{i}} \hat{p}_0\left(e_{i}+\epsilon\nabla\phi_{e_{i}}^{\epsilon}\right)\right)\nonumber\\
		&-\mathbf{d}^{\epsilon}(\hat{p}_0+\eta_{r}\epsilon\sum_{i=1}^{d}(\partial_{x_{i}} \hat{p}_0)\phi_{e_{i}}^\epsilon)+\Lambda(\hat{p}_0+\eta_{r}\epsilon\sum_{i=1}^{d}(\partial_{x_{i}} \hat{p}_0)\phi_{e_{i}}^\epsilon).
	\end{align}
	Regarding the first term on the right-hand side of (\ref{w shizi}), from (\ref{hat q2}), we get 
	\begin{align}\label{b chuli}
		\partial_{t} \hat{p}_0	=&\sum_{i=1}^{d}\nabla(\eta_{r}\partial_{x_{i}}\hat{p}_0)\cdot\bar{\mathbf{a}}e_{i}+\nabla\cdot((1-\eta_{r})\bar{\mathbf{a}}\nabla \hat{p}_0)+\sum_{i=1}^{d}(\eta_{r}\partial_{x_{i}}\hat{p}_0)\bar{\mathbf{b}}e_{i}\nonumber\\
        &+(1-\eta_{r})\mathbf{\bar{b}}(\nabla  \hat{p}_0)+(\mathbf{\bar{d}}-\Lambda)\hat{p}_0.
	\end{align}

		Then from (\ref{elli unifor})-(\ref{boundedness1}),(\ref{w shizi}) and (\ref{eta}) and using identification of $ {\tilde{H}^{-1}_{\mathrm{par}}} $ \cite[Lemma 3.11]{armstrong2018quantitative}, we have
	\begin{align}\label{t1 t2 t3 t4}
		&\left\Vert(\partial_{t}-\nabla\cdot\mathbf{a}^{\epsilon}\nabla-\mathbf b^\epsilon\nabla-\mathbf d^\epsilon+\Lambda)w^{\epsilon}\right\Vert_{L^{2}(H^{-1}(U); I)}\nonumber\\
		\le& C \left(\sum_{i=1}^{d}\epsilon\Vert\phi_{e_{i}}^{\epsilon}\partial_{t}(\eta_{r}\partial_{x_{i}} \hat{p}_0)\Vert_{L^{2}(V)}\nonumber\right.\\
        &
		+\sum_{i=1}^{d}\left\lVert\nabla(\eta_{r}\partial_{x_{i}} \hat{p}_0)\cdot\left(\mathbf{a}^{\epsilon}\left(e_{i}+\epsilon\nabla\phi_{e_{i}}^{\epsilon}\right)-\bar{\mathbf{a}}e_{i}\right)\right\lVert_{{\tilde{H}}^{-1}_{\mathrm{par}}(V)}\nonumber\\
  &+\sum_{i=1}^{d}\epsilon\Vert\phi_{e_{i}}^{\epsilon}\nabla(\eta_{r}\partial_{x_{i}} \hat{p}_0)\Vert_{L^{2}(V)}+\Vert(1-\eta_{r})\nabla  \hat{p}_0\Vert_{L^{2}(V)}
	\nonumber \\
		&+\sum_{i=1}^{d}\left\lVert\eta_{r}\partial_{x_{i}} \hat{p}_0\cdot\left(\mathbf{b}^{\epsilon}\left(e_{i}+\epsilon\nabla\phi_{e_{i}}^{\epsilon}\right)-\bar{\mathbf{b}}e_{i}\right)\right\lVert_{{\tilde{H}}^{-1}_{\mathrm{par}}(V)}\nonumber\\
 &\left.+\left\lVert(\mathbf{d}^{\epsilon}-\bar{\mathbf{d}}) \hat{p}_0\right\lVert_{{\tilde{H}}^{-1}_{\mathrm{par}}(V)}+\sum_{i=1}^{d}\epsilon\Vert\phi_{e_{i}}^{\epsilon}\eta_{r}\partial_{x_{i}} \hat{p}_0 \Vert_{L^{2}(V)}\right).
	\end{align}
 Let us set the seven items on the right-hand side of (\ref{t1 t2 t3 t4}) as $T_1, T_2, T_3, T_4, T_5, T_6,$ and $T_7$ in order.
By using (\ref{da E}), (\ref{eta}), and (\ref{point}), we have
\begin{align}\label{t1 t2}
T_1+T_2\le&C\sum_{i=1}^{d}\left(\Vert\partial_{t}(\eta_r\partial_{x_{i}}\hat{p_0})\Vert_{L^{\infty}(V)}\epsilon\Vert\phi_{e_{i}}^{\epsilon}\Vert_{L^{2}(V)}\nonumber\right.\\
&\left.+\Vert\nabla(\eta_r\partial_{x_{i}}\hat{p_0})\Vert_{W_{\mathrm{par}}^{1,\infty}(V)}\left\Vert\mathbf{a}^{\epsilon}\left(e_{i}+\epsilon\nabla\phi_{e_{i}}^{\epsilon}\right)-\mathbf{\bar{a}}e_{i}\right\Vert_{ {\tilde{H}}^{-1}_{\mathrm{par}}(V)}\right)\nonumber\\
\le&Cr^{-4-d/2}\mathcal{E}(\epsilon)\Vert f\Vert_{W_{\mathrm{par}}^{1,2+\delta}(V)},
\end{align}
and
	\begin{align}\label{t5}
		T_5+T_6 \le&C\sum_{i=1}^{d}\Vert\eta_r\partial_{x_{i}} \hat{p}_0\Vert_{W_{\mathrm{par}}^{1,\infty}(V)}\left\Vert\mathbf{b}^{\epsilon}\left(e_{i}+\epsilon\nabla\phi_{e_{i}}^{\epsilon}\right)-\mathbf{\bar{b}}e_{i}\right\Vert_{{\tilde{H}}^{-1}_{\mathrm{par}}(V)}\nonumber\\
  &+C\Vert \hat{p}_0\Vert_{W_{\mathrm{par}}^{1,\infty}(V)}\left\Vert\mathbf{d}^{\epsilon}-\mathbf{\bar{d}}\right\Vert_{{\tilde{H}}^{-1}_{\mathrm{par}}(V)}\nonumber\\
		\le&Cr^{-4-d/2}\mathcal{E}(\epsilon)\Vert f\Vert_{W_{\mathrm{par}}^{1,2+\delta}(V)}.
	\end{align}
	The calculations for $T_3$, $T_4$, and $T_7$ are similar to \cite[Theorem 1.17]{armstrong2019quantitative}, so we have
 \begin{align}\label{t3 t4}
T_3+T_4+T_7\le C\left(r^{\frac{\delta}{4+2\delta}}+r^{-4-d/2}\mathcal{E}(\epsilon)\right)\Vert f\Vert_{W_{\mathrm{par}}^{1,2+\delta}(V)}.
\end{align}
 The difference is that we need to use the Meyers estimate (Lemma \ref{meyer}), which is for the parabolic equation with lower-order terms.
 
 Combining (\ref{t1 t2 t3 t4})-(\ref{t3 t4}), we finally get (\ref{step 1}). \end{proof}
 
\begin{proof}[Proof of Theorem \ref{quali0}]
Combining (\ref{comebine with step1}) with (\ref{step 1}), it yields
\begin{align}\label{step 2}
		\Vert \hat{p}^\epsilon- w^{\epsilon}\Vert_{L^{2}(H^{1}(U); I)}\le C\left(r^{\frac{\delta}{4+2\delta}}+r^{-4-d/2}\mathcal{E}(\epsilon)\right)\Vert f\Vert_{W_{\mathrm{par}}^{1,2+\delta}(V)}.
	\end{align}
 Observing that
 \begin{align*}
\Vert\nabla \hat{p}^\epsilon-\nabla w^\epsilon\Vert_{{\tilde{H}}^{-1}_{\mathrm{par}}(V)}
\le C\Vert\nabla \hat{p}^\epsilon-\nabla w^\epsilon\Vert_{{L}^{2}(V)}
\le C\Vert \hat{p}^\epsilon-w^\epsilon\Vert_{L^{2}(H^{1}(U); I)},
	\end{align*}
we conclude 
  \begin{align}\label{step 2 1}
\Vert\nabla \hat{p}^\epsilon-\nabla w^\epsilon\Vert_{{\tilde{H}}^{-1}_{\mathrm{par}}(V)}
\le C\left(r^{\frac{\delta}{4+2\delta}}+r^{-4-d/2}\mathcal{E}(\epsilon)\right)\Vert f\Vert_{W_{\mathrm{par}}^{1,2+\delta}(V)}.
	\end{align}

Now it remains to prove
\begin{align}\label{step 3}
		\Vert w^\epsilon- \hat{p}_0\Vert_{{L}^{2}(V)}+\Vert\nabla w^\epsilon-\nabla \hat{p}_0\Vert_{{\tilde{H}}^{-1}_{\mathrm{par}}(V)}
		\le C\left(r^{\frac{\delta}{4+2\delta}}+r^{-4-d/2}\mathcal{E}(\epsilon)\right)\Vert f\Vert_{W_{\mathrm{par}}^{1,2+\delta}(V)}.
	\end{align}
To show (\ref{step 3}), we only need to use the definition of $ w^\epsilon $, which is consistent with the equation without the lower-order term. Therefore, we do not provide the detailed process here. For more details, the reader  could refer to \cite[Theorem 1.1]{armstrong2018quantitative}.

By using (\ref{step 2})-(\ref{step 3}) and the triangle inequality, we obtain
\begin{align*}
		&\Vert \hat{p}^\epsilon- \hat{p}_0\Vert_{L^{2}(V)}+\Vert\nabla \hat{p}^\epsilon-\nabla \hat{p}_0\Vert_{{\tilde{H}}^{-1}_{\mathrm{par}}(V)}\\
		\le &\Vert \hat{p}^\epsilon- w^\epsilon\Vert_{L^{2}(V)}+\Vert w^\epsilon- \hat{p}_0\Vert_{L^{2}(V)}+\Vert\nabla \hat{p}^\epsilon-\nabla w^\epsilon\Vert_{{\tilde{H}}^{-1}_{\mathrm{par}}(V)}+\Vert\nabla w^\epsilon-\nabla \hat{p}_0\Vert_{{\tilde{H}}^{-1}_{\mathrm{par}}(V)}\\
		\le	& C\left(r^{\frac{\delta}{4+2\delta}}+r^{-4-d/2}\mathcal{E}(\epsilon)\right)\Vert f\Vert_{W_{\mathrm{par}}^{1,2+\delta}(V)}.
	\end{align*}
 
Taking $\beta(\delta):=\delta/(4+2\delta) $, we conclude (\ref{result 1}).
\end{proof}

\begin{proof}[Proof of Theorem \ref{qualitative result}]
Since $\hat{p}^\epsilon=\exp{(-\Lambda t)}p^\epsilon$ and $\hat{p}_0=\exp{(-\Lambda t)}p_0$, by (\ref{result 1}), we get
\begin{align*}
		\Vert p^\epsilon- {p}_0\Vert_{L^{2}(V)}+\Vert\nabla {p}^\epsilon-\nabla {p}_0\Vert_{{\tilde{H}}^{-1}_{\mathrm{par}}(V)}
		\le	 C\left(r^{\frac{\delta}{4+2\delta}}+r^{-4-d/2}\mathcal{E}(\epsilon)\right)\Vert f\Vert_{W_{\mathrm{par}}^{1,2+\delta}(V)}.
	\end{align*}
Combining  Proposition \ref{convergence L to H} with (\ref{weak conv}), we get, $\mathbb{P}$-almost surely,
$
\limsup_{\epsilon\to 0} \mathcal{E}(\epsilon)\to 0.
$
Finally, there exists a small $r$ such that (\ref{quali result}) holds.
\end{proof}
Armstrong and Smart \cite{armstrong2016quantitative} expand the method of Dal Maso and Modica \cite{dal1986nonlinear} to obtain quantitative results. Inspired by this, in future work, we may estimate $ \mathcal{E}(\epsilon)$ in this study  to get the speed of convergence. 

\section*{Acknowledgments}
This work was supported by the WISE program (MEXT) at Kyushu University.
I would like to extend my sincerest gratitude to Professor Osada Hirofumi.



\end{document}